\documentclass[12pt]{amsart}

\usepackage[margin=1in]{geometry}
\usepackage{hyperref}
\hypersetup{
	colorlinks=true,
	linkcolor=blue,
	citecolor=blue}

\usepackage[utf8]{inputenc}
\usepackage{comment}
\usepackage{hyperref}
\usepackage{amsthm}
\usepackage{amsmath}
\usepackage{graphicx}
\usepackage{amsfonts,amssymb,tikz,mathrsfs}
\usepackage{amsmath}
\hypersetup{
     colorlinks=true,
     linkcolor=blue,
     citecolor=blue}
\usepackage{tikz}
\usepackage{dsfont}
\usepackage{marginnote}
\usetikzlibrary{matrix, fit}
\usepackage{amssymb, shuffle}
\usepackage{caption}
\usepackage{ytableau}

\newtheorem{theorem}{Theorem}[section]
\newtheorem{corollary}[theorem]{Corollary}

\newtheorem{example}[theorem]{Example}
\newtheorem{proposition}[theorem]{Proposition}

\newtheorem{definition}[theorem]{Definition}

\numberwithin{equation}{section}

\newcommand{\fS}{\mathfrak{S}}

\newcommand{\NSym}{\mathrm{NSym}}

\newcommand{\NCSym}{\mathrm{NCSym}}

\newcommand{\QSym}{\mathrm{QSym}}
\newcommand{\Sym}{\mathrm{Sym}}

\newcommand{\rmO}{\mathrm{O}}
\newcommand{\rmT}{\mathrm{T}}
\newcommand{\rmS}{\mathrm{S}}
\newcommand{\cV}{\mathcal{V}}
\newcommand{\sfS}{\mathsf{S}}
\newcommand{\Sn}{\mathfrak{S}}

\newcommand{\rmE}{\mathrm{E}}

\newcommand{\OtE}{\mathrm{OtE}}
\newcommand{\EtO}{\mathrm{EtO}}
\newcommand{\img}{\mathrm{img}}

\title{Extending the descent-to-peak map and its applications}
\author{Farid Aliniaeifard}
\address{
	Research Center for Mathematics and Interdisciplinary Sciences, Shandong University \\
	Frontiers Science Center for Nonlinear Expectations, Ministry of Education \\
	Qingdao, Shandong, 266237, P. R. China}
\email{farid@sdu.edu.cn}

\author{Shu Xiao Li}
\address{
	Research Center for Mathematics and Interdisciplinary Sciences, Shandong University \\
	Frontiers Science Center for Nonlinear Expectations, Ministry of Education \\
	Qingdao, Shandong, 266237, P. R. China}
\email{lishuxiao@sdu.edu.cn}

\thanks{Both authors were supported in part by the Provincial Nature Science Foundation of Shandong, Project No. ZR2024QA026 and the Fundamental Research Funds for the Central Universities.}


\date{}

\begin{document}

\maketitle

\begin{abstract}
The descent-to-peak map serves as a bridge between algebra and combinatorics. We use it as a tool for proving the equidistribution of peak and valley sets of standard Young tableaux with a very short argument. We also introduce a new shuffle basis of quasisymmetric functions whose elements are eigenvectors of the descent-to-peak map. Using this basis, we then extend the notion of the peak algebra and of the descent-to-peak map to shuffle, tensor, and symmetric algebras.  
	 
\end{abstract}

\tableofcontents

\section{Introduction}

The tools that help us translate or transfer results from one branch of mathematics to another are indispensable. One such tool is the descent-to-peak map, which we use to give a very short proof of the main result in \cite{Z24}. Extending this map to other Hopf algebras could serve as a powerful method for further connecting algebra and combinatorics. Moreover, the image of the descent-to-peak map is called the peak algebra $\Pi$, originally defined by Stembridge using enriched $P$-partitions. To extend this notion to other Hopf algebras, we need a new basis that not only includes the eigenvectors of the descent-to-peak map but also exhibits nice combinatorial properties. We construct such a basis.

A special family of symmetric functions, known as the Schur's $Q$ functions, indexed by odd partitions, are introduced in \cite{S11} to study the projective representations of symmetric and alternating groups. Combinatorially, the Schur's $Q$ functions are equipped with a theory of shifted tableaux, including RSK correspondence, Littlewood-Richardson rule and jeu de taquin \cite{S87, S89, W84}. Moreover, the subspace spanned by Schur's $Q$ functions, $\Omega$, forms a self-dual Hopf algebra \cite{M95}.

Using the combinatorics of enriched $P$-partitions, the Schur's $Q$ functions are generalized to the peak algebra $\Pi$, indexed by odd compositions, in the fundamental paper \cite{S97}. The peak algebra and its graded dual have triggered a variety amount of studies. For instance, it is shown that the peak algebra corresponds to the representation of $0$-Hecke-Clifford algebra. Also, the dual peak algebra is a Hopf ideal in noncommutative symmetric functions with respect to the internal product. See \cite{B03, B02, B04, B14, H07, N03} for a small portion.


There are already many trials to generalize peak algebra in the literature. Most of these generalizations are combinatorial through enriched $P$-partition and peak sets. Notably, in \cite{J15}, the peak algebra is generalized to the Poirier-Reutenauer Hopf algebra of standard Young tableaux, which is introduced in \cite{P95}, and the authors also extend the notion of the peak algebra to noncommuting variables \cite{A25}. Other generalizations can be found in \cite{A04, B06, H06}.

In this paper, we extend the peak algebra from an algebraic point of view. As observed in \cite{A06}, there is another description of $\Omega$ and $\Pi$, given by character theory of combinatorial Hopf algebras. The Hopf algebra of quasisymmetric functions $\QSym$ and the canonical character
$$\begin{array}{cccc}
\zeta_\QSym: & \QSym & \to & \mathbb{C}\\
& f(x_1,x_2,\dots) & \mapsto & f(1,0,0,\dots)
\end{array}$$
give the terminal object in the category of combinatorial Hopf algebras. The \emph{descent-to-peak} map $\Theta_\QSym$ is the unique Hopf morphism that makes the following diagram commute.
\begin{center}
	\begin{tikzpicture}
	\node (Q1) at (0,0) {$\QSym$};
	\node (Q2) at (3,0) {$\QSym$};
	\node (Q) at (1.5,-1.5) {$\mathbb{C}$};
	
	\draw[black, ->] (Q1) to (Q2);
	\draw[black, ->] (Q1) to (Q);
	\draw[black, ->] (Q2) to (Q);
	
	\node at (1.5,0.3) {$\Theta_\QSym$};
	\node at (2.8,-1) {$\zeta_\QSym$};
	\node at (-0.2,-1) {$\overline{\zeta_\QSym^{-1}}\zeta_\QSym$};
	\end{tikzpicture}
\end{center}
The image of $\Theta_\QSym$ is exactly $\Pi$. Similarly, $\Omega$ can be constructed as the image of $\Theta_\Sym$.

For any combinatorial Hopf algebra $\mathsf{H}$ with a morphism $\Phi:\mathsf{H}\to\QSym$, we define a \emph{theta map} for $\mathsf{H}$ to be a graded Hopf morphism making the following diagram commute. 
$$
\begin{tikzpicture}
\node(TV) at (0,0){$\mathsf{H}$};
\node(Sym) at (4,0){$\QSym$};
\node(OV) at (0,-2){$\mathsf{H}$};
\node(Omega) at (4,-2){$\QSym$};

\draw[thick,->]  (TV)->(Sym);
\draw[thick,->]  (OV)->(Omega);
\draw[thick,->]   (TV)->(OV);
\draw[thick,->]   (Sym)->(Omega);

\node(Phi) at (2,0.2){$\Phi$};
\node(Phi) at (2,-1.8){$\Phi$};
\node(thetaV) at (-0.3,-1){$\Theta$};
\node(thetaV) at (4.7,-1){$\Theta_\QSym$};
\end{tikzpicture} 
$$
Therefore, the descent-to-peak map of $\QSym$ is a theta map of $\QSym$,  $\Theta_{\QSym}$. We refer to the descent-to-peak map of $\QSym$ as the theta map of $\QSym$. The image $\Theta(\mathsf{H})$ is called a peak algebra for $\mathsf{H}$. In order to find the theta maps, we introduce a shuffle basis $\{S_\alpha\}$ of $\QSym$, indexed by compositions. The elements of this basis satisfy three properties:
\begin{enumerate}
	\item Their product and coproduct are the shuffle product and the deconcatenation coproduct of compositions, respectively.
	\item They are eigenvectors of $\Theta_\QSym$.
	\item $\{S_\alpha: \text{all parts of $\alpha$ are odd\}}$ is a basis for $\Pi$.
\end{enumerate}
Using this new basis and its dual basis, we are able to find explicit theta maps and peak algebras for shuffle algebras, tensor algebras, and symmetric algebras. The symmetric functions in noncommuting variables, $\NCSym$, can be viewed as a tensor algebra, as a consequence of \cite{L11}, which finds a set of free generators of $\NCSym$ that are also primitive. Our construction results in a peak algebra for $\NCSym$ indexed by set partitions whose parts all have odd size. In \cite{A25}, the authors also construct a peak algebra of $\NCSym$ combinatorially through non-commutative enriched $P$-partitions, which is different from the one in this paper. 

There is another large family of shuffle and tensor algebras that our construction can be applied to. As shown in \cite{A05}, for any Hopf algebra $\mathsf{H}$ that is graded as a coalgebra, the associated graded Hopf algebra, $\mathrm{gr}(\mathsf{H})$, is a shuffle algebra. The graded dual of $\mathrm{gr}(\mathsf{H})$ is a tensor algebra. We use the graded associated Hopf algebra on permutations, defined in \cite{L21,V14}, as an example and describe a theta map and peak algebra for it.

This paper is organized as follows. Section \ref{sec:prelim} gives necessary backgound information on the Hopf algebras symmetric functions and quasisymmetric functions, together with character theory of combinatorial Hopf algebras and precise definitions of theta maps. Our shuffle basis is introduced and studied in Section \ref{sec:basis}. In Section \ref{sec:main}, we give a very short proof of the main result in \cite{Z24} using the descent-to-peak map.  In Section \ref{sec:Hopfs}, we revisit the shuffle, tensor and symmetric algebra. We then define candidates of theta maps and give explicit formula. The Hopf subalgebras of odd elements are introduced as the peak algebras. And lastly, Section \ref{sec:perms} gives examples of theta maps and peak algebras, including the symmetric functions in noncommuting variables, and the graded associated Hopf algebra on permutations.

\section{Preliminaries}\label{sec:prelim}

In this section we present the Hopf algebras $\Sym, \QSym, \NSym$, together with the character theory of combinatorial Hopf algebras. For more details about Hopf algebras we refer the readers to \cite{G18}. The base field is set to be $\mathbb{C}$ for convenience.

\subsection{The Hopf algebras $\QSym, \Sym$ and $\NSym$}

A {\it composition} $\alpha$ of a positive integer $n$, written $\alpha \models n$, is a finite ordered list of positive integers $(\alpha_1,\ldots,\alpha_\ell)$ where $\alpha_1+\dots+\alpha_\ell=n$. We let $\ell(\alpha)=\ell$ be the number of parts of $\alpha$ and it is called the {\it length} of $\alpha$. 
When there is no confusion, we sometimes write $\alpha$ as a word $\alpha_1\alpha_2\cdots\alpha_{\ell(\alpha)}$. The {\it concatenation} of $\alpha=(\alpha_1,\ldots,\alpha_{\ell(\alpha)}) \models n$ and $\beta=(\beta_1,\ldots,\beta_{\ell(\beta)})\models m$  is the composition $\alpha  \beta=(\alpha_1,\ldots,\alpha_{\ell(\alpha)},\beta_1,\ldots,\beta_{\ell(\beta)})$.

For a positive integer $n$, let $[n]:=\{1,2,\ldots,n \}$. There is a one-to-one correspondence $\mathcal{I}$ between the set of all compositions of $n$ and the set of all subsets of $[n-1]$ such that  for $\alpha \models n$,
$$\mathcal{I}(\alpha)=\{ \alpha_1,\alpha_1+\alpha_2, \ldots,\alpha_1+\cdots+\alpha_{\ell(\alpha)-1} \}.$$
Let $\alpha,\beta\models n$. We write $\beta \leq \alpha$, or $\alpha$ refines $\beta$ if $\mathcal{I}(\beta)\subseteq \mathcal{I}(\alpha)$.\\

Let $\Sn_n$ denote the set of permutations of $\{1,2,\dots,n\}$. Let ${\rm Sh}_{n,m}$ be the following subset of the permutation group ${\Sn}_{n+m}$,
$${\rm Sh}_{n,m}=\{\sigma\in {S}_{n+m}:\sigma^{-1}(1)<\sigma^{-1}(2)<\cdots<\sigma^{-1}(n);\sigma^{-1}(n+1)<\cdots<\sigma^{-1}(n+m)\}.$$
For compositions $\alpha$ and $\beta$ with lengths $n$ and $m$, respectively, we define $\alpha\shuffle \beta$ to be the multi-set $$\{(\gamma_{\sigma(1)},\gamma_{\sigma(2)},\dots,\gamma_{\sigma(n+m)}):\sigma\in {\rm Sh}_{n,m}\}$$ where $\gamma=(\gamma_1,\gamma_2,\dots,\gamma_{n+m})$ is the concatenation $\alpha\beta$.

The quasishuffle $\alpha\overline{\shuffle}\beta$ is the multi-set of compositions $(\delta_1,\dots,\delta_\ell)$ such that
\begin{enumerate}
	\item $(\delta_1,\dots,\delta_\ell)\leq(\gamma_1,\dots,\gamma_{n+m})$ for some $(\gamma_1,\dots,\gamma_{n+m})\in\alpha\shuffle\beta$,
	\item either $\delta_i=\gamma_j$ for some $j$, or $\delta_i=\gamma_j+\gamma_{j+1}$ where $\gamma_j$ is a part of $\alpha$ and $\gamma_{j+1}$ is a part of $\beta$.
\end{enumerate}

For a composition $\alpha$ of $n$, the {\it monomial quasisymmetric function} $M_{\alpha}$ is
$$M_\alpha:=\sum_{i_1<\ldots<i_{\ell(\alpha)}  }x_{i_1}^{\alpha_1} \cdots x_{i_{\ell(\alpha)}}^{\alpha_{\ell(\alpha)}}$$ 
which is  an element of the  algebra of  bounded-degree formal power series in commutative variables
$\{x_i\}_{i\geq 1}$. By convention, $M_{()}= 1$ where $()$ is the unique composition of $0$.
The vector space of {\it quasisymmetric functions} is denoted by $\QSym$ and is defined as
$$\QSym=\bigoplus_{n\geq 0} \QSym_n$$
where
$$\QSym_n=\mathbb{C}\text{-span}\{ M_\alpha: \alpha \models n \}.$$
The space of quasisymmetric functions is indeed a Hopf algebra where its product is given by the quasishuffle
$$M_\alpha\cdot M_\beta=\sum_{\gamma\in\alpha\overline{\shuffle}\beta}M_\gamma.$$
and its coproduct is given by deconcatenation
$$\Delta(M_\alpha)=\sum_{\beta \gamma=\alpha}M_\beta \otimes M_\gamma.$$ 

A special subspace of $\QSym$ is the Hopf algebra of {\it symmetric functions}, $\Sym$. The space $\Sym$ is (commutatively) freely generated by the {\it scaled power sum functions} defined as
$$p_n=\frac{1}{n}M_{n}.$$
The scaled power sum functions are primitive elements i.e. for $n\geq 1$, $\Delta(p_n)=1\otimes p_n+p_n\otimes 1$. For a composition $\alpha$, we write $p_\alpha=p_{\alpha_1}\cdots p_{\alpha_{\ell(\alpha)}}$. In particular, $\Sym$ is both commutative and cocommutative.

The complete homogeneous functions of $\Sym$ are defined as
$$h_n=\sum_{\alpha\models n}M_\alpha$$
and we write $h_{\alpha}=h_{\alpha_1}\cdots h_{\alpha_{\ell(\alpha)}}$.

There is a well-known expansion of $p_n$ in terms of the complete homogeneous functions, cf. \cite[Ex. 5.43]{G18}, given by
$$p_n=\frac{1}{n}\sum_{\alpha\models n}(-1)^{\ell(\alpha)-1}\alpha_{\ell(\alpha)}h_\alpha.$$

The Hopf algebra of {\it non-commutative symmetric functions}, $\NSym$, originally defined in \cite{G95}, is the graded dual of $\QSym$. Let $\{H_\alpha\}$ be the basis, indexed by compositions, of $\NSym$ dual to the basis $\{ M_\alpha\}$ of $\QSym$. Thus, We have the Hopf pairing $\langle H_\alpha,M_\beta\rangle=\delta_{\alpha,\beta}$ where $\delta_{\alpha,\beta}$ is the Kronecker delta.

Equivalently, $\NSym = \langle H_1, H_2, \cdots \rangle$ is the non-commutative algebra freely generated by the set $\{H_n\}_{n\geq1}$ with coproduct defined, on generators and extended linearly and multiplicatively, as
$$\Delta(H_n)=\sum_{i+j=n}H_i\otimes H_j.$$

For a composition $\alpha$, we write $H_\alpha=H_{\alpha_1}\cdots H_{\alpha_{\ell(\alpha)}}$. The Hopf algebra $\NSym$ is graded with $\deg(H_n)=n$. A linear basis for the homogeneous component $\NSym_n$ of degree $n$ is $\{H_\alpha:\alpha\models n\}$.

It is well-known that $\Sym$ is self-dual via the isomorphism defined on generators as
$$\begin{array}{cccc}
	I:&\Sym^*&\rightarrow& \Sym\\
	&m_n^*& \mapsto & h_n.
\end{array}$$
It can be shown, via the Cauchy identity, that $I(p_n^*)=np_n$. Let $\iota:\Sym\to\QSym$ be the natural embedding. There is a forgetful projection $\pi:\NSym\to\Sym$ defined on generators as $H_n\mapsto h_n$. The map $\pi$ is a graded Hopf morphism and is also the map composition $I^{-1}\circ\iota^*$ where $\iota^*:\NSym\to\Sym^*$ is the adjoint map of $\iota$ i.e. $\pi=I\circ\iota^*=(\iota\circ I)^*$.

\subsection{Combinatorial Hopf algebras}
A \emph{character} for a Hopf algebra is a linear and multiplicative map $\zeta:\mathsf{H}\to\mathbb{C}$. A \emph{combinatorial Hopf algebra} is a pair $(\mathsf{H}, \zeta)$ where $\mathsf{H}$ is a graded connected Hopf algebra and $\zeta$ is a character for $\mathsf{H}$. The set of characters $\rm{Hom}(\mathsf{H},\mathbb{C})$ forms a group under \emph{convolution product}
$$\zeta\zeta' = m\circ(\zeta\otimes\zeta')\circ\Delta.$$
The identity element of the group $\mathrm{Hom}(\mathsf{H},\mathbb{C})$ is the counit $\epsilon_{\mathsf{H}}$ and inverse element is given by $\zeta^{-1}=\zeta\circ \mathcal{S}_{\mathsf{H}}$ where $\mathcal{S}_{\mathsf{H}}$ is the antipode of $\mathsf{H}$. We define $\bar{\zeta}$ to be the character such that $\bar{\zeta}|_{\mathsf{H}_n}=(-1)^n\zeta|_{\mathsf{H}_n}$. 

A {\it morphism of combinatorial Hopf algebras} $\Phi: (\mathsf{H},\zeta_{\mathsf{H}})\rightarrow (\mathsf{H}^{'},\zeta_{\mathsf{H}^{'}})$, is a graded Hopf morphism $\Phi:\mathsf{H}\to\mathsf{H}'$ such that $\zeta_{\mathsf{H}^{'}}\circ \Phi=\zeta_{\mathsf{H}}$. The character theory of Hopf algebras has been studied by Aguiar, Bergeron, and Sottile in \cite{A06}.

Let $\zeta_\QSym$ be the character of $\QSym$ given by

$$\zeta_\QSym(f(x_1,x_2,\dots))=f(1,0,0,0,\dots).$$

\begin{example}
	We have $\zeta_\QSym(p_n)=1$ and $\zeta_\QSym(M_\alpha)=\left\{\begin{array}{cc}
		1 & \text{if }\ell(\alpha)=0\text{ or }1\\
		0 & \text{otherwise.}
		\end{array}\right.$
\end{example}

Let $\zeta_\Sym$ be the restricted character $\zeta_\QSym|_\Sym$. The characters $\zeta_\QSym$ and $\zeta_\Sym$ are canonical characters in the following sense.

\begin{theorem}{\rm \label{thm:ABSmain}\cite[Theorem 4.1]{A06}}
	For any combinatorial Hopf algebra $(\mathsf{H},\zeta)$, there exists a unique morphism of combinatorial Hopf algebras
	$$\Phi: (\mathsf{H},\zeta)\rightarrow (\QSym,\zeta_{\QSym}).$$
	Moreover, if $\mathsf{H}$ is cocommutative, then $\Phi(\mathsf{H})\subseteq\Sym$.
\end{theorem}

The \textit{odd and even Hopf subalgebras} of a combinatorial Hopf algebra $(\mathsf{H},\zeta_\mathsf{H})$, denoted by $\sfS_-(\mathsf{H},\zeta_\mathsf{H})$ and $\sfS_+(\mathsf{H},\zeta_\mathsf{H})$, respectively, are defined as the largest sub-coalgebra contained in $\ker\left(\overline{\zeta_{\mathsf{H}}^{-1}}-\zeta_{\mathsf{H}}\right)$ and $\ker\left(\overline{\zeta_{\mathsf{H}}}-\zeta_{\mathsf{H}}\right)$, respectively.

We will be using the following important properties of the odd and even Hopf subalgebras.

\begin{theorem}{\rm \label{thm:ABSodd}\cite[Theorem 5.3]{A06}}
	The coalgebras $\sfS_-(\mathsf{H},\zeta_\mathsf{H})$ and $\sfS_+(\mathsf{H},\zeta_\mathsf{H})$ are graded Hopf sub-algebras of $\mathsf{H}$. Moreover, a homogeneous element $h\in\mathsf{H}$ belongs to $\sfS_-(\mathsf{H},\zeta_\mathsf{H})$ (or $\sfS_+(\mathsf{H},\zeta_\mathsf{H})$) if and only if $(\rm{id}\otimes \left(\overline{\zeta_{\mathsf{H}}^{-1}}-\zeta_{\mathsf{H}}\right)\otimes\rm{id})\circ \Delta^{(2)}(h)=0$ (or $(\rm{id}\otimes \left(\overline{\zeta_{\mathsf{H}}}-\zeta_{\mathsf{H}}\right)\otimes\rm{id})\circ \Delta^{(2)}(h)=0$, respectively).
\end{theorem}

\subsection{Peak algebras and the theta maps}

Considering the combinatorial Hopf algebra  $(\QSym,\overline{\zeta_\QSym^{-1}}\zeta_\QSym)$, let $\Theta_\QSym$ be the unique combinatorial Hopf morphism
$$\Theta_\QSym:(\QSym,\overline{\zeta_\QSym^{-1}}\zeta_\QSym)\to(\QSym,\zeta_\QSym).$$
The map $\Theta_\QSym$ is originally studied by Stembridge using enriched $P$-partitions \cite{S97} and reconstructed in \cite{A06}. The image of $\Theta_\QSym$ is known as the {\it peak algebra} of $\QSym$, denoted by $\Pi$.

Let $\Theta_\Sym$ be the unique combinatorial Hopf morphism
$$\Theta_\Sym:(\Sym,\overline{\zeta_\Sym^{-1}}\zeta_\Sym)\to(\Sym,\zeta_\Sym).$$
The map $\Theta_{\Sym}=\Theta_{\QSym}|_{\Sym}$, and moreover, it maps $h_n$ to $q_n$ where $q_n$ is the Schur's $Q$ function cf. \cite{M95, S11}. It can be shown that $\Theta_\Sym$ can be equivalently defined on scaled power sum functions, and extended multiplicatively, as
$$\Theta_\Sym(p_n)=\left\{\begin{array}{cc}
2p_n & \text{if }n \text{ is odd}\\
0& \text{otherwise.}
\end{array}\right.$$
The image of $\Theta_\Sym$ is known as the space of Schur's $Q$ functions, and denoted by $\Omega$.

The images $\Pi$ and $\Omega$ are precisely the odd Hopf subalgebras $\sfS_{-}(\QSym,\zeta_\QSym)$ and $\sfS_{-}(\Sym,\zeta_\Sym)$, respectively.

In this paper, we attempt to extend the notion of $\Theta_\QSym$ and peak algebras to other combinatorial Hopf algebras, in the following sense.

\begin{definition}
	Let $(\mathsf{H},\zeta)$ be a combinatorial Hopf algebra, and let $\Phi:(\mathsf{H},\zeta)\to(\QSym,\zeta_\QSym)$ be the unique combinatorial Hopf morphism. A {\it theta map} for $(\mathsf{H},\zeta)$ is a map $\Theta:\mathsf{H}\to\mathsf{H}$ that makes the following diagram commutes
	
	$$
	\begin{tikzpicture}
	\node(TV) at (0,0){$\mathsf{H}$};
	\node(Sym) at (4,0){$\QSym$};
	\node(OV) at (0,-2){$\mathsf{H}$};
	\node(Omega) at (4,-2){$\QSym$};
	
	\draw[thick,->]  (TV)->(Sym);
	\draw[thick,->]  (OV)->(Omega);
	\draw[thick,->]   (TV)->(OV);
	\draw[thick,->]   (Sym)->(Omega);
	
	\node(Phi) at (2,0.2){$\Phi$};
	\node(Phi) at (2,-1.8){$\Phi$};
	\node(thetaV) at (-0.3,-1){$\Theta$};
	\node(thetaV) at (4.7,-1){$\Theta_\QSym$};
	\end{tikzpicture} 
	$$
	
	The image ${\rm img}(\Theta)$ is called a peak algebra for $(\mathsf{H},\zeta)$.
\end{definition}

Since $\img(\Theta_\QSym)=\Pi$, we sometimes use the alternative diagram
$$
\begin{tikzpicture}
\node(TV) at (0,0){$\mathsf{H}$};
\node(Sym) at (4,0){$\QSym$};
\node(OV) at (0,-2){$\Theta(\mathsf{H})$};
\node(Omega) at (4,-2){$\Pi$};

\draw[thick,->]  (TV)->(Sym);
\draw[thick,->]  (OV)->(Omega);
\draw[thick,->]   (TV)->(OV);
\draw[thick,->]   (Sym)->(Omega);

\node(Phi) at (2,0.2){$\Phi$};
\node(Phi) at (2,-1.8){$\Phi$};
\node(thetaV) at (-0.3,-1){$\Theta$};
\node(thetaV) at (4.7,-1){$\Theta_\QSym$};
\end{tikzpicture} 
$$
When $\mathsf{H}$ is cocommutative, by Theorem \ref{thm:ABSmain}, we have $\Phi(\mathsf{H})\subseteq\Sym$ and hence the following diagram is commutative. 
$$
\begin{tikzpicture}
\node(TV) at (0,0){$\mathsf{H}$};
\node(Sym) at (4,0){$\Sym$};
\node(OV) at (0,-2){$\Theta(\mathsf{H})$};
\node(Omega) at (4,-2){$\Omega$};

\draw[thick,->]  (TV)->(Sym);
\draw[thick,->]  (OV)->(Omega);
\draw[thick,->]   (TV)->(OV);
\draw[thick,->]   (Sym)->(Omega);

\node(Phi) at (2,0.2){$\Phi$};
\node(Phi) at (2,-1.8){$\Phi$};
\node(thetaV) at (-0.3,-1){$\Theta$};
\node(thetaV) at (4.6,-1){$\Theta_\Sym$};
\end{tikzpicture} 
$$
In particular, $\Pi$ and $\Omega$ are peak algebras for $(\QSym,\zeta_\QSym)$ and $(\Sym,\zeta_\Sym)$, respectively. Additionally, $\Theta_\QSym$ and $\Theta_\Sym$ are theta maps for $(\QSym,\zeta_\QSym)$ and $(\Sym,\zeta_\Sym)$, respectively.

\section{The peak and valley statistics are equidistributed}\label{sec:main}

In this section, we give a very short proof for the main result in \cite{Z24}. 

Let $\lambda$ be a partition. Its \emph{Young diagram} is a left-justified array of $n$ squares with $\lambda_1$ squares in the first row, $\lambda_2$ in the second row, and so on. The \emph{conjugate} of $\lambda$, denoted $\lambda^t$, is the partition whose Young diagram is obtained by reflecting the Young diagram of $\lambda$ across its main diagonal.  A \emph{square} $(i,j)$ in a Young diagram of shape $\lambda$ is the square in the $i$-th column and $j$-th row, where rows are numbered from top to bottom and columns from left to right.

Let $\mu$ and $\lambda$ be Young diagrams with $\mu \subseteq \lambda$ (i.e., $\mu_i \leq \lambda_i$ for all $i$), and assume $|\lambda| - |\mu| = n$. The \emph{skew diagram} $\lambda/\mu$ is defined by removing the Young diagram $\mu$ from the top-left corner of $\lambda$. The integer $n$ is called the size of $\lambda/\mu$. We will identify $\lambda/\emptyset$ with $\lambda$.

A standard Young tableau (SYT) $T$ of skew shape $\lambda/\mu$ is a filling of the skew diagram $\lambda/\mu$ with the numbers $1,2,\ldots,n$ such that entries are increasing along each row and each column. We borrow the following example from \cite[page 3]{Z24}.  Figure \ref{fig:young} (left) illustrates an SYT of shape $(4,3,2)$, and Figure 1 (right) illustrates an SYT of shape $(5,4,2)/(2,1)$. Let $\operatorname{SYT}(\lambda/\mu)$ denote the set of SYT of shape $\lambda/\mu$.

\begin{center}
\begin{figure}[h]
\centering
\[
\begin{ytableau}
1 & 4 & 6 & 9 \\
2 & 5 & 8 \\
3 & 7
\end{ytableau}
\quad \quad \quad \quad \quad \quad 
\begin{ytableau}
\none & \none & 1 & 2 & 6 \\
\none & 3 & 4 & 7 \\
5 & 8
\end{ytableau}
\]
\caption{Examples of standard Young tableaux.}\label{fig:young}
\end{figure}
\end{center}

Let $T$ be a standard Young tableau (SYT) of shape $\lambda/\mu$ with $n$ entries. An index $i$ ($1 \leq i \leq n-1$) is called a \emph{descent} of $T$ if $i+1$ appears in a lower row of $T$ than $i$; otherwise, $i$ is called an \emph{ascent} of $T$. For $2 \leq i \leq n-1$, we call $i$ a \emph{peak} of $T$ if $i-1$ is an ascent and $i$ is a descent, and we call $i$ a \emph{valley} of $T$ if $i-1$ is a descent and $i$ is an ascent. Denote by $\operatorname{Peak}(T)$ and $\operatorname{Val}(T)$ the sets of peaks and valleys of $T$, respectively. Note that
$$\operatorname{Peak}(T)=\{i\in[2,n-1]: i\in \operatorname{Des}(T),i-1\not\in \operatorname{Des}(T) \},$$ and 
$$\operatorname{Val}(T)=\{i\in[2,n-1]: i\not\in \operatorname{Des}(T),i-1\in \operatorname{Des}(T) \}.$$
For example, if $T$ is the SYT shown in Figure 1 (left), then $\operatorname{Peak}(T) = \{4,6\}$ and $\operatorname{Val}(T) = \{3,5,7\}$.

Let $B$ be a peak set, define
$$K_{B}=\sum_{A\subseteq [n-1]\atop B\subseteq A \cup (A+1)}M_{\mathcal{I}^{-1}(A)}.$$
Then $$\{ K_B: B\subseteq [2,n-1] \text{~is a peak set} \}$$ is a basis for the Peak algebra $\Pi$. Let $A\subseteq [n-1]$. Define 
$$F_A = \sum_{C\subseteq A} M_{\mathcal{I}^{-1}(C)}.$$
The complement of $A$ is given by $A^c=[n-1]\setminus A$. Define
$$
\begin{array}{cccc}
\psi: &\QSym& \rightarrow& \QSym\\
& F_A & \mapsto & F_{A^c}.
\end{array}
$$
Consider that $\psi(h_n)=e_n$. Therefore, when we restrict $\psi$ to $\Sym$, it gives the well-known omega involtution \cite[Equation (2. 4.15) ]{G18}, 
$$
\begin{array}{cccc}
\omega=\psi|_{\Sym}: &\Sym& \rightarrow& \Sym\\
& h_n & \mapsto & e_n \\
& e_n & \mapsto & h_n \\
& s_{\lambda/\mu} & \mapsto & s_{\lambda^t/\mu^t}. 
\end{array}
$$
The \emph{peak set} of $A$ is given by $\operatorname{Peak}(A) = \{i \in [2, n-1] : i \in A, i-1 \notin A\}$, its \emph{valley set} is given by $\operatorname{Val}(A) = \{i \in [2, n-1] : i \notin A, i-1 \in A\}$. In particular,
\[
\operatorname{Peak}(A) = \operatorname{Val}(A^c).
\]
For example, if $A=\{2,4,5,8\}\subseteq [10-1]$, then $\operatorname{Peak}(A)=\{2,4,8\}$ and $A^c=\{1,3,6,7,9\}$, and then $ \operatorname{Val}(A^c)=\{2,4,8\}$.
By  \cite[Equation (9.1)]{A25}, we have
\begin{equation}
\Theta_{\Sym}(F_A) = K_{\operatorname{Peak}(A)}.
\end{equation}

Let $\mathcal{A} = \{A_1, A_2, \ldots, A_k\}$ be a multi-set of subsets of  $[n-1]$. The {peak statistics} of $\mathcal{A}$ is given by the multi-set $\operatorname{Peak}(\mathcal{A}) = \{\operatorname{Peak}(A_i) : A_i \in \mathcal{A}\}$, and the {valley statistics} of $\mathcal{A}$ is given by the multi-set $\operatorname{Val}(\mathcal{A}) = \{\operatorname{Val}(A_i) : A_i \in \mathcal{A}\}$.

We define the quasisymmetric generating function of $\mathcal{A}$ to be
\[
F_{\mathcal{A}} = \sum_{A_i \in \mathcal{A}} F_{A_i},
\]
 the quasisymmetric generating function of $\operatorname{Peak}(\mathcal{A})$ is given by
\begin{equation}\label{eq:gfpeak}
\Theta_{\QSym}(F_{\mathcal{A}}) = \sum_{A_i \in \mathcal{A}} K_{\operatorname{Peak}(A_i)},
\end{equation}
and the quasisymmetric generating function of $\operatorname{Val}(\mathcal{A})$ is given by
\begin{equation}\label{eq:gfval}
\Theta_{\QSym}(\psi(F_{\mathcal{A}})) = \sum_{A_i \in \mathcal{A}} K_{\operatorname{Peak}(A_i^c)} = \sum_{A_i \in \mathcal{A}} K_{\operatorname{Val}(A_i)}.
\end{equation}
Therefore, by Equations \eqref{eq:gfpeak} and \eqref{eq:gfval}, we have that 
$$\sum_{A_i \in \mathcal{A}} K_{\operatorname{Peak}(A_i)}=\sum_{A_i \in \mathcal{A}} K_{\operatorname{Val}(A_i)}\quad\text{if and only if}\quad \Theta_{\QSym}(F_{\mathcal{A}}) = \Theta_{\QSym}(\psi(F_{\mathcal{A}})).$$

We now give a very short proof of \cite[Theorem 1.1]{Z24}, which is the main result of that paper. 
For a set $A$, write $t^A=\prod_{i\in A} t_i$. Set $t^{A}=1$ if $A=\emptyset$.  

\begin{theorem}\label{thm:main}{\rm \cite[Theorem 1.1]{Z24}} 
Let $n\geq 1$. Then for any skew diagram $\lambda/\mu$ of size $n$, the set-valued statistics $\operatorname{Peak}$ and $\operatorname{Val}$ are equidistributed over $\operatorname{SYT}(\lambda/\mu)$, that is, 
$$\sum_{T\in \operatorname{SYT}(\lambda/\mu)}t^{\operatorname{Peak}(T)}=\sum_{T\in \operatorname{SYT}(\lambda/\mu)}t^{ \operatorname{Val}(T)}.$$
\end{theorem}
\begin{proof} Note that by \cite[Chapter III, Section 8, Example 10]{M95}, we have $\Theta_{\QSym}(e_n)=\Theta_{\QSym}(h_n).$ Therefore, $$\Theta_{\QSym} (\omega(e_n))=\Theta_{\QSym}(h_n)=\Theta_{\QSym}(e_n),$$ thus $\Theta_{\QSym}|_{\Sym}=\Theta_{\QSym}\circ \psi|_{\Sym}.$

For any $\lambda/\mu$, we have 
 $$\Theta_{\QSym}(s_{\lambda/\mu})=\Theta_{\QSym}(\psi(s_{\lambda/\mu})).$$ Let  $\mathcal{A} = \{\operatorname{Des}(T) : T \in \operatorname{SYT}(\lambda)\}$. Then 
\[
F_{\mathcal{A}} = \sum_{T \in \operatorname{SYT}(\lambda/\mu)} F_{\operatorname{Des}(T)} = s_{\lambda/\mu}.
\]
Since 
$$\Theta_{\QSym}(F_{\mathcal{A}}) = \Theta_{\QSym}(\psi(F_{\mathcal{A}})),$$ we have 
$$\sum_{T \in \operatorname{SYT}(\lambda/\mu)} K_{\operatorname{Peak}(\operatorname{Des}(T))}=\sum_{T \in \operatorname{SYT}(\lambda/\mu)} K_{\operatorname{Val}(\operatorname{Des}(T))}.$$
Consider that $$\operatorname{Peak}(\operatorname{Des}(T))=\{i\in[2,n-1]: i\in \operatorname{Des}(T),i-1\not\in \operatorname{Des}(T) \}=\operatorname{Peak}(T),$$ and 
$$\operatorname{Val}(\operatorname{Des}(T))=\{i\in[2,n-1]: i\not\in \operatorname{Des}(T),i-1\in \operatorname{Des}(T) \}=\operatorname{Val}(T).$$ Therefore, 
$$\sum_{T \in \operatorname{SYT}(\lambda/\mu)} K_{\operatorname{Peak}(T)}=\sum_{T \in \operatorname{SYT}(\lambda/\mu)} K_{\operatorname{Val}(T)}.$$
Now applying the map $K_{\operatorname{Peak}(T)}\mapsto t^{{\operatorname{Peak}(T)}}$ to the the both sides of the above equation gives 
$$\sum_{T\in \operatorname{SYT}(\lambda/\mu)}t^{\operatorname{Peak}(T)}=\sum_{T\in \operatorname{SYT}(\lambda/\mu)}t^{ \operatorname{Val}(T)}.$$
%
%
\end{proof}

\section{The shuffle basis of $\QSym$}\label{sec:basis}
 
In this section, we define our new basis of $\QSym$ and study its properties. This basis is a shuffle basis and also its elements are the eigenvectors of $\Theta_\QSym$. 

\begin{definition}
	For a composition $\alpha=(\alpha_1,\dots,\alpha_\ell)$,
	\begin{itemize}
		\item its odd-to-even set, $\OtE(\alpha)$ is defined to be $\{i:\alpha_i\text{ is odd},\alpha_{i+1}\text{ is even}\}$,
		\item its even-to-odd set, $\EtO(\alpha)$ is defined to be $\{i:\alpha_i\text{ is even},\alpha_{i+1}\text{ is odd}\}$,
		\item if $\OtE(\alpha)=\{i_1,i_2,\dots,i_k\}$ with $i_1<i_2<\cdots<i_k$, the odd-min composition of $\alpha$ is defined to be
		$$m_o(\alpha)=(\alpha_1+\cdots+\alpha_{i_1},\alpha_{i_1+1}+\cdots+\alpha_{i_2},\dots,\alpha_{i_k+1}+\cdots+\alpha_{\ell}),$$
		\item given a composition $\beta=(\beta_1,\dots,\beta_p)$ such that $m_o(\alpha)\leq\beta\leq\alpha$, then each $\beta_i$ is a sum of consecutive parts of $\alpha$. Assume $\beta_i=\alpha_{j}+\alpha_{j+1}+\cdots+\alpha_{j+k}$ for some $j,k$. Let $\rmO_\alpha^\beta(i)$ and $\rmE_\alpha^\beta(i)$ be the number of odd and even parts in $(\alpha_j,\alpha_{j+1},\dots,\alpha_{j+k})$, respectively. We define the coefficients
		$$c_{\alpha}^{\beta}(i)=\cfrac{1}{\rmO_\alpha^\beta(i)!\rmE_\alpha^\beta(i)!},$$
		and
		$$c_\alpha^\beta=\prod_i c_\alpha^\beta(i).$$
	\end{itemize}
\end{definition}

\begin{example}
	If $\alpha=34421332$, then $\OtE(\alpha)=\{1,7\}$, $\EtO(\alpha)=\{4\}$ and $m_o(\alpha)=(3,17,2)$. Suppose $\beta=(3,8,9,2)=(3,4+4,2+1+3+3,2)$, then
	\begin{itemize}
		\item $\rmO_\alpha^\beta(1)=1$, $\rmE_\alpha^\beta(1)=0$ and $c_\alpha^\beta(1)=1$,
		\item $\rmO_\alpha^\beta(2)=0$, $\rmE_\alpha^\beta(2)=2$ and $c_\alpha^\beta(2)=1/2$,
		\item $\rmO_\alpha^\beta(3)=3$, $\rmE_\alpha^\beta(3)=1$ and $c_\alpha^\beta(3)=1/6$,
		\item $\rmO_\alpha^\beta(4)=0$, $\rmE_\alpha^\beta(4)=1$ and $c_\alpha^\beta(4)=1$,
	\end{itemize}
	and overall, $c_\alpha^\beta=1/12$.
\end{example}

\begin{definition}
	The shuffle function $S_\alpha$, indexed by a composition $\alpha\models n$, is defined as
	$$S_\alpha=\sum_{m_o(\alpha)\leq\beta\leq\alpha}c_\alpha^\beta M_\beta.$$
\end{definition}

\begin{proposition}
	The set $\{S_\alpha:\alpha\models n\}$ of shuffle functions gives a basis of $\QSym_n$.
\end{proposition}

\begin{proof}
	It is clear from the definition that $c_\alpha^\alpha=1$ for all compositions $\alpha$. Then the proposition follows from the triangularity of the transition matrix from the shuffle basis to the monomial basis.
\end{proof}

\begin{theorem}\label{thm:shuffle-product}
	The product of the shuffle basis is given by the shuffle of compositions, i.e.
	$$S_\alpha\cdot S_\beta=\sum_{\gamma\in\alpha\shuffle\beta}S_\gamma.$$
\end{theorem}

\begin{proof}
	For each composition $\gamma\in\alpha\shuffle\beta$, and $m_o(\gamma)\leq\delta\leq\gamma$, we construct a pair of compositions $\hat{\alpha},\hat{\beta}$ and $\epsilon\in\hat{\alpha}\overline{\shuffle}\hat{\beta}$ as follows. By definition, $\delta_i=\alpha_{p_i}+\alpha_{p_i+1}+\cdots+\alpha_{p_i+q}+\beta_{r_i}+\beta_{r_i+1}+\cdots+\beta_{r_i+t}$ for some $p_i,q,r_i,t$ such that the multi-set $\{\alpha_{p_i},\dots,\alpha_{p_i+q},\beta_{r_i},\beta_{r_i+t}\}$ is equal to $\{\gamma_{s_i},\gamma_{s_i+1},\dots,\gamma_{s_i+q+t}\}$ for some $s_i$ and $\OtE(\gamma_{s_i},\dots,\gamma_{s_i+q+t})=\emptyset$.
	
	We set $a_i=\alpha_{p_i}+\cdots+\alpha_{p_i+q}$ and $b_i=\beta_{r_i}+\cdots+\beta_{r_i+t}$. Then, $\hat{\alpha}$ and $\hat{\beta}$ are the compositions obtained from $(a_1,\dots,a_{\ell(\delta)})$ and $(b_1,\dots,b_{\ell(\delta)})$, respectively by removing zeros. Since $\OtE(\alpha_{p_i},\dots,\alpha_{p_i+q})=\OtE(\beta_{r_i},\dots,\beta_{r_i+t})=\emptyset$, we have $m_o(\alpha)\leq\hat{\alpha}\leq\alpha$ and $m_o(\beta)\leq\hat{\beta}\leq\beta$. Furthermore, $\epsilon$ is obtained from the shuffle $(a_1,b_1,\dots,a_{\ell(\delta)},b_{\ell(\delta)})$, removing possible zeros, then merged to $(a_1+b_1,\dots,a_{\ell(\delta)}+b_{\ell(\delta)})$. As compositions, we have $\epsilon=\delta$.
	
	This process gives a map
	$$Y:\{(\gamma,\delta):\gamma\in\alpha\shuffle\beta,m_o(\gamma)\leq\delta\leq\gamma\}\to\{(\hat{\alpha},\hat{\beta},\epsilon):m_o(\alpha)\leq\hat{\alpha}\leq\alpha,m_o(\beta)\leq\hat{\beta}\leq\beta,\epsilon\in\hat{\alpha}\overline{\shuffle}\hat{\beta}\}.$$
	For example, for $\alpha={\color{red}2113}$, $\beta={\color{blue}122}$, $\gamma={\color{blue}12}{\color{red}2}{\color{blue}2}{\color{red}113}$ and $\delta=1623$, then $a_1=0,a_2=2,a_3=2,a_4=3$, $b_1=1,b_2=4,b_3=0,b_4=0$, $\hat{\alpha}=223$, $\hat{\beta}=14$ and $\epsilon=1623$.
	
	The map $Y$ is surjective. From $\hat{\alpha},\hat{\beta}$ and $\epsilon$, we can recover $a_i$ and $b_i$ from the  quasishuffle. Suppose $a_i=\alpha_{p_i}+\cdots+\alpha_{p_i+q}$, $b_i=\beta_{r_i}+\cdots+\beta_{r_i+t}$ and $\EtO(\alpha_{p_i},\dots,\alpha_{p_i+q})=\{x\}$, $\EtO(\beta_{r_i},\dots,\beta_{r_i+t})=\{y\}$ (take $x=q$ and $y=t$ when non-existing), then one pre-image is $$\gamma=(\dots,\alpha_{p_i},\dots,\alpha_{p_i+x},\beta_{r_i},\dots,\beta_{r_i+y},\alpha_{p_i+x+1},\dots,\alpha_{p_i+q},\beta_{r_i+y+1},\dots,\beta_{r_i+t},\dots).$$
	In general, let $d_\epsilon(i)=\binom{x+y}{x}\binom{q+t-x-y}{q-x}$ and let $d_\epsilon=\prod_id_\epsilon(i)$. Then the number of pre-images is $d_\epsilon$.
	
	Taking the example above, there are totally three $\gamma$ that yields the same result, namely ${\color{blue}1}{\color{red}2}{\color{blue}22}{\color{red}113}$, ${\color{blue}12}{\color{red}2}{\color{blue}2}{\color{red}113}$ and ${\color{blue}122}{\color{red}2113}$.
	
	It is not hard to see that $c_\gamma^\delta=d_\epsilon c_\alpha^{\hat{\alpha}}c_\beta^{\hat{\beta}}$. Therefore, the statement of this theorem follows from the identities
	\begin{align*}
		\sum_{\gamma\in\alpha\shuffle\beta}S_\gamma=\sum_{\gamma\in\alpha\shuffle\beta}\sum_{m_o(\gamma)\leq\delta\leq\gamma}c_{\gamma}^{\delta}M_\delta=\sum_{m_o(\alpha)\leq\hat{\alpha}\leq\alpha \atop m_o(\beta)\leq\hat{\beta}\leq\beta}\sum_{\epsilon\in\hat{\alpha}\overline{\shuffle}\hat{\beta}}c_{\alpha}^{\hat{\alpha}}c_{\beta}^{\hat{\beta}}M_\epsilon=\sum_{m_o(\alpha)\leq\hat{\alpha}\leq\alpha \atop m_o(\beta)\leq\hat{\beta}\leq\beta}c_{\alpha}^{\hat{\alpha}}c_{\beta}^{\hat{\beta}}M_{\hat{\alpha}}\cdot M_{\hat{\beta}}=S_\alpha\cdot S_\beta.
	\end{align*}
\end{proof}

\begin{theorem}\label{thm:shuffle-coproduct}
	The coproduct of the shuffle basis is given by the deconcatenation of the compositions, i.e.
	$$\Delta(S_\alpha)=\sum_{\beta\gamma=\alpha}S_\beta\otimes S_\gamma.$$
\end{theorem}

\begin{proof}
	Fix a composition $\alpha\models n$. Since $\QSym$ is graded, it suffices to show that for each $i$,
	$$\Delta_{i,n-i}(S_\alpha)=\sum_{\beta\gamma=\alpha \atop \beta\models i}S_{\beta}\otimes S_{\gamma}.$$
	For simplicity, we use $\Delta_i$ to denote $\Delta_{i,n-i}$.
	
	Fix an $i$. If there is no $p$ such that $\alpha_1+\cdots+\alpha_p=i$, then for any $\beta\leq\alpha$, there is no $q$ such that $\beta_1+\cdots+\beta_q=i$, which implies $\Delta_i(M_\beta)=0$. In this case,
	$$\Delta_i(S_\alpha)=\sum_{m_o(\alpha)\leq\beta\leq\alpha}c_\alpha^\beta\Delta_i(M_\beta)=0.$$
	
	Assume that $\alpha_1+\cdots+\alpha_p=i$ for some $p$, let ${}^i\alpha=(\alpha_1,\dots,\alpha_p)$ and $\alpha^i=(\alpha_{p+1},\dots,\alpha_{\ell(\alpha)})$, then
	\begin{align*}
		\Delta_i(S_\alpha)&=\sum_{m_o(\alpha)\leq\beta\leq\alpha}c_\alpha^\beta\Delta_i(M_\beta)\\
		&=\sum_{m_o(\alpha)\leq\beta\leq\alpha \atop \beta_1+\cdots+\beta_q=i}c_\alpha^\beta M_{(\beta_1,\dots,\beta_q)}\otimes M_{(\beta_{q+1},\dots,\beta_{\ell(\beta)})}\\
		&=\sum_{m_o(\alpha)\leq\beta\leq\alpha \atop \beta_1+\cdots+\beta_q=i}\left( M_{(\beta_1,\dots,\beta_q)}\prod_{i=1}^q c_\alpha^\beta(i)\right)\otimes\left( M_{(\beta_{q+1},\dots,\beta_{\ell(\beta)})}\prod_{i=q+1}^{\ell(\beta)} c_\alpha^\beta(i)\right)\\
		&=\sum_{m_o({}^i\alpha)\leq \gamma\leq {}^i\alpha\atop m_o(\alpha^i)\leq\mu\leq\alpha^i}\left(c_{{}^i\alpha}^\gamma M_\gamma\right)\otimes \left(c_{\alpha^i}^\mu M_\mu\right).
	\end{align*}
	
	The last equality follows from the fact that $m_o(\alpha)\leq\beta\leq\alpha$ and $\beta_1+\cdots+\beta_q=i$ if and only if $m_o({}^i\alpha)\leq {}^i\beta\leq {}^i\alpha$ and $m_o(\alpha^i)\leq\beta^i\leq\alpha^i$. Moreover, it is straight forward from the definition that $c_{{}^i\alpha}^{{}^i\beta}=\prod_{i=1}^qc_\alpha^\beta(i)$ and $c_{\alpha^i}^{\beta^i}=\prod_{i=q+1}^{\ell(\beta)}c_\alpha^\beta(i)$.
\end{proof}

By the Hopf structure given above, the shuffle basis makes $\QSym$ a shuffle algebra. And the antipode formula for shuffle basis follows from the well-known antipode formula of shuffle algebras \cite{BS14}.

\begin{corollary}
	Let $\alpha=(\alpha_1,\dots,\alpha_\ell)$. The antipode on the shuffle basis is given by
	$$\mathcal{S}(S_\alpha)=(-1)^{\ell}S_{(\alpha_\ell,\dots,\alpha_1)}.$$
\end{corollary}

Recall that a composition $\alpha=(\alpha_1,\dots,\alpha_\ell)$ is odd if each $\alpha_i$ is an odd number.

\begin{theorem}\label{thm:theta}
	The map $\Theta_\QSym$ acts on the shuffle basis as follows
	$$\Theta_\QSym(S_\alpha)=\left\{\begin{array}{cc}
		2^{\ell(\alpha)}S_\alpha & \text{if }\alpha \text{ is odd}\\
		0& \text{otherwise.}
		\end{array}\right.$$
\end{theorem}

\begin{proof}
	Let $\Phi:\QSym\to\QSym$ be the linear map defined on the shuffle basis as
	$$\Phi(S_\alpha)=\left\{\begin{array}{cc}
	2^{\ell(\alpha)}S_\alpha & \text{if }\alpha \text{ is odd}\\
	0& \text{otherwise.}
	\end{array}\right.$$
	It is easy to check that $\Phi$ is a Hopf algebra morphism.
	Since $\Theta_\QSym$ is the unique combinatorial Hopf morphism from $(\QSym,\overline{\zeta_\QSym^{-1}}\zeta_\QSym)$ to $(\QSym,\zeta_\QSym)$, it suffices to show that $\overline{\zeta_\QSym^{-1}}\zeta_\QSym=\zeta_\QSym\circ\Phi$.
	
	Recall that
	$$\zeta_\QSym(M_\alpha)=\left\{\begin{array}{cc}
	1 & \text{if }\ell(\alpha)=0\text{ or }1\\
	0 & \text{otherwise}.
	\end{array}\right.$$
	It follows from the definition of shuffle functions that,
	\begin{equation}\label{eq:3.1}
		\zeta_\QSym(S_\alpha)=\left\{\begin{array}{cc}
			\cfrac{1}{p!(\ell(\alpha)-p)!} & \text{if }\EtO(\alpha)=\{p\}\text{ and }\OtE(\alpha)=\emptyset\\
			\cfrac{1}{\ell(\alpha)!} & \text{if }\EtO(\alpha)=\OtE(\alpha)=\emptyset\\
			0 & \text{otherwise}.
		\end{array}\right.
	\end{equation}
	Hence, we obtain the following formula for $\zeta_\QSym\circ\Phi$
	$$\zeta_\QSym\circ\Phi(S_\alpha)=\left\{\begin{array}{cc}
	\cfrac{2^{\ell(\alpha)}}{\ell(\alpha)!}& \text{if }\alpha\text{ is odd}\\
	0 & \text{otherwise.}
	\end{array}\right.$$
	Combining with the antipode formula yields
	\begin{equation}\label{eq:3.2}
		\overline{\zeta_\QSym^{-1}}(S_\alpha)=\left\{\begin{array}{cc}
			(-1)^{|\alpha|+\ell(\alpha)}\cfrac{1}{q!(\ell(\alpha)-q)!} & \text{if }\EtO(\alpha)=\emptyset\text{ and }\OtE(\alpha)=\{q\}\\
			(-1)^{|\alpha|+\ell(\alpha)}\cfrac{1}{\ell(\alpha)!} & \text{if }\EtO(\alpha)=\OtE(\alpha)=\emptyset\\
			0 & \text{otherwise}.
		\end{array}\right.
	\end{equation}
	To sum up,
	\begin{align*}
	&\overline{\zeta_\QSym^{-1}}\zeta_\QSym(S_\alpha)=\sum_{i=0}^{\ell(\alpha)}\overline{\zeta_\QSym^{-1}}(S_{(\alpha_1,\dots,\alpha_i)})\zeta_\QSym(S_{(\alpha_{i+1},\dots,\alpha_{\ell(\alpha)})})\\
	=&\left\{\begin{array}{cc}
	\displaystyle\sum_{i=q}^p(-1)^{\alpha_1+\cdots+\alpha_i+i}\frac{1}{q!(i-q)!}\frac{1}{(p-i)!(\ell(\alpha)-p)!} & \text{if }\EtO(\alpha)=\{p\},\OtE(\alpha)=\{q\}\text{ and }q<p\\
	\displaystyle\sum_{i=0}^p(-1)^{\alpha_1+\cdots+\alpha_i+i}\frac{1}{i!}\frac{1}{(p-i)!(\ell(\alpha)-p)!} & \text{if }\EtO(\alpha)=\{p\}\text{ and }\OtE(\alpha)=\emptyset\\
	\displaystyle\sum_{i=q}^{\ell(\alpha)}(-1)^{\alpha_1+\cdots+\alpha_i+i}\frac{1}{q!(i-q)!}\frac{1}{(\ell(\alpha)-i)!} & \text{if }\EtO(\alpha)=\emptyset\text{ and }\OtE(\alpha)=\{q\}\\
	\displaystyle\sum_{i=0}^{\ell(\alpha)}(-1)^{\alpha_1+\cdots+\alpha_i+i}\frac{1}{i!}\frac{1}{(\ell(\alpha)-i)!} & \text{if }\EtO(\alpha)=\OtE(\alpha)=\emptyset\\
	0 & \text{otherwise.}
	\end{array}\right.\\
	\end{align*}
	In the first case, the summation is alternating in sign because $\alpha_i$ are even for $q< i\leq p$. Therefore, the summation can be simplified to
	$$\frac{(-1)^q}{q!(\ell(\alpha)-p)!}\sum_{i=q}^p(-1)^{i}\frac{1}{(i-q)!(p-i)!}.$$
	It is easy to see that this summation must vanish, e.g., multiply it by $(p-q)!$ and use the binomial identity.
	
	The same argument can be applied to all other cases except the sub-case of 4 that $\alpha$ is an odd composition. When $\alpha$ is odd, all terms in the sum have a positive sign. Hence, we have
	
	$$\overline{\zeta_\QSym^{-1}}\zeta_\QSym(S_\alpha)=\sum_{i=0}^{\ell(\alpha)}\frac{1}{i!(\ell(\alpha)-i)!}=\frac{2^{\ell(\alpha)}}{\ell(\alpha)!}.$$
	
	Comparing with the formula of $\zeta_\QSym\circ\Phi$ completes the proof.
\end{proof}

Let $\{S_\alpha^*\}$ denote the basis of $\NSym$ that is dual to the shuffle basis of $\QSym$ i.e. we have the Hopf pairing $\langle S_\alpha^*,S_\beta\rangle=\delta_{\alpha,\beta}$. The basis $\{S_\alpha^*\}$ is called the {\it dual shuffle basis}. It is easy to see that $\NSym$ is freely generated by $\{S_1^*,S_2^*,\dots\}$ and $S_n^*$ is primitive for all $n$.

\begin{theorem}\label{thm:dual-shuffle}
	The forgetful projection $\pi$ maps the dual shuffle basis to the scaled power sum basis. More precisely, for all compositions $\alpha$,
	$$\pi(S_\alpha^*)=p_\alpha.$$
\end{theorem}

\begin{proof}
	Since both $\{S_\alpha^*\}$ and $\{p_\alpha\}$ are multiplicative, it suffices to show that $\pi(S_n^*)=p_n$.
	
	Since $\pi=(\iota\circ I)^*$, we have $\langle p_\alpha^*,\pi(S_n^*)\rangle=\langle S_n^*,\iota\circ I(p_\alpha^*)\rangle=\langle S_n^*,(\prod_i\alpha_i)p_\alpha \rangle$.
	
	Recall that $p_n=\cfrac{1}{n}M_n=\cfrac{1}{n}S_n$, hence, $\langle p_n^*,\pi(S_n^*)\rangle =\langle S_n^*,np_n\rangle=1$. Moreover, if $\alpha$ has more than one parts, then $\langle S_n^*, p_\alpha\rangle=\cfrac{1}{\prod_i\alpha_i}\langle S_n^*,S_{\alpha_1}\cdots S_{\alpha_{\ell(\alpha)}}\rangle=0$ because the shuffle product of the elments in the $S$-basis does not contain the term $S_n$. Then, we have
	$$\langle p_\alpha^*,\pi(S_n^*)\rangle=\left\{\begin{array}{cc}
		1 & \text{if }\alpha=(n)\\
		0 & \text{otherwise}.
	\end{array}\right.$$
	Therefore, $\pi(S_n^*)=\displaystyle\sum_\alpha\langle p_\alpha^*,\pi(S_n^*)\rangle p_\alpha=p_n$.
\end{proof}

\section{The shuffle, tensor, and symmetric algebras}\label{sec:Hopfs}

In this section, we revisit shuffle, tensor, and symmetric algebras. Since the shuffle basis $\{S_\alpha: \alpha \text{~is a composition}\}$  gives an eigen basis for $\QSym$, and moreover, $\{S_\alpha: \alpha \text{~is an odd composition}\}$ is a basis for $\Pi$, we can easily extend $\Theta_\QSym$ to shuffle, tensor, and symmetric algebras.  

\subsection{The shuffle algebra}

Let $V$ be a vector space with a countable basis and
$$V^{\otimes n}=\underbrace{V\otimes V \otimes \dots \otimes V}_{\text{$n$ times}}.$$

As a vector space, the shuffle algebra on $V$ is
$$\rmS(V)=\bigoplus_{n\geq 0}V^{\otimes n}.$$

We fix a basis $A=\{v_1,v_2,\dots\}$ of $V$ throughout this paper. For convenience, we write $v_1\otimes v_2\otimes \cdots \otimes v_n$ as $v_1v_2\cdots v_n$. 
Then a basis for $\rmS(V)$ contains all elements of the form $v_{\alpha_1}v_{\alpha_2}\cdots v_{\alpha_\ell}$. For a composition $\alpha=(\alpha_1,\dots,\alpha_\ell)$, we write $v_{\alpha_1}v_{\alpha_2}\cdots v_{\alpha_\ell}$ as $v_\alpha$. Essentially, the elements of $\rmS(V)$ are linear combinations of the words on the alphabet $A$. Note that $V^{\otimes 0}$ is generated by the empty word.

We now present the Hopf structure of $\rmS(V)$ on the basis elements, and extend it linearly.

The product of $\rmS(V)$ is given by the shuffle of the words, that is,

$$v_\alpha\cdot v_\beta=\sum_{\gamma\in\alpha\shuffle\beta}v_\gamma.$$

The coproduct is given by the deconcatenation of words, that is, 
$$\Delta(v_\alpha)=\sum_{i=0}^{\ell(\alpha)} v_{(\alpha_1,\dots,\alpha_i)}\otimes v_{(\alpha_{i+1},\dots,\alpha_{\ell(\alpha)})}.$$ 

It is well-known, e.g. in \cite{BS14}, that the antipode formula for $S(V)$ is given by
$$\mathcal{S}(v_{(\alpha_1,\dots,\alpha_\ell(\alpha))})=(-1)^{\ell(\alpha)}v_{(\alpha_{\ell(\alpha)},\dots,\alpha_1)}.$$

When $V$ is equipped with a grading such that the basis elements $v_i$ are homogeneous, the degree of $v_{(\alpha_1,\dots,\alpha_n)}$ in $V^{\otimes n}$ is 
$${\rm deg}(v_{\alpha_1})+{\rm deg}(v_{\alpha_2})+\dots+{\rm deg}(v_{\alpha_n}).$$
With such a grading, $\rmS(V)$ becomes a graded connected Hopf algebra. In the rest of this paper, we always assume that $V$ has such a grading and moreover, the set $\{\alpha:\deg(v_\alpha)=n\}$ is finite for all $n$.

Consider the following graded Hopf algebras morphism, $$\begin{array}{cccc}
	\Phi_{\rmS(V)}:&\rmS(V)&\rightarrow& \QSym\\
	&v_{(\alpha_1,\dots,\alpha_\ell)}& \mapsto & S_{(\deg(v_{\alpha_1}),\dots,\deg(v_{\alpha_\ell}))}.
\end{array}$$
The fact that $\Phi_{\rmS(V)}$ is a graded Hopf morphism follows directly from the Hopf structure of the shuffle basis in Theorem \ref{thm:shuffle-product} and \ref{thm:shuffle-coproduct}.

The \textit{canonical} character of $\rmS(V)$, denoted by $\zeta_{\rmS(V)}$ is defined to be the map composition
$$\zeta_{\rmS(V)}=\zeta_{\QSym}\circ\Phi_{\rmS(V)}.$$

More explicitly, let $\alpha=(\alpha_1,\dots,\alpha_\ell)$ and $(\beta_1,\dots,\beta_\ell)=(\deg(v_{\alpha_1}),\dots,\deg(v_{\alpha_\ell}))$, then
$$\zeta_{\rmS(V)}(v_\alpha)=\left\{\begin{array}{cc}
\cfrac{1}{p!(\ell-p)!} & \text{if }\EtO(\beta)=\{p\}\text{ and }\OtE(\beta)=\emptyset\\
\cfrac{1}{\ell!} & \text{if }\EtO(\beta)=\OtE(\beta)=\emptyset\\
0 & \text{otherwise}.
\end{array}\right..$$

In particular, $\zeta_{\rmS(V)}(v_i)=1$ for all $i$.

\subsection{The tensor algebra}
As a vector space, the tensor algebra $\rmT(V)$ is the same as $\rmS(V)$. The Hopf structure of $\rmT(V)$ can be expressed, on basis elements and extended linearly, as follows.

The product is given by the concatenation of compositions, that is,
$$v_\alpha\cdot v_\beta=v_{\alpha\beta}.$$

The coproduct is given by the deshuffle of composition, that is,
$$\Delta(v_{(\alpha_1,\dots,\alpha_\ell)})=\sum_{J\subseteq[\ell]}v_{\alpha_J}\otimes v_{\alpha_{[\ell]\setminus J}}$$
where $[\ell]=\{1,2,\dots,\ell\}$, and if $J=\{i_1,i_2,\dots,i_k\}$ with $i_1<i_2<\cdots<i_k$, then $\alpha_J=(\alpha_{i_1},\dots,\alpha_{i_k})$.

We can see that $v_i$'s are primitive elements that freely generate $\rmT(V)$.

The antipode formula is given by
$$\mathcal{S}(v_{(\alpha_1,\dots,\alpha_\ell(\alpha))})=(-1)^{\ell(\alpha)}v_{(\alpha_{\ell(\alpha)},\dots,\alpha_1)}.$$

The tensor algebra and shuffle algebra are graded dual of each other under the Hopf pairing $\langle -,-\rangle:\rmS(V)\times\rmT(V)\to\mathbb{C}$ given by $\langle v_\alpha,v_\beta\rangle=\delta_{\alpha,\beta}$ where $\delta_{\alpha,\beta}$ is the Kronecker delta.

Consider the following graded Hopf algebra morphism,
$$\begin{array}{cccc}
\Phi_{\rmT(V)}:&\rmT(V)&\rightarrow& \Sym\\
 & v_i & \mapsto & p_{\deg(v_i)}.
\end{array}$$
where $p$ is the scaled power sum symmetric function. The fact that $\Phi_{\rmT(V)}$ is a graded Hopf morphism follows from the fact that the scaled power sum functions $p_n$ are both multiplicative and primitive.

The \textit{canonical} character of $\rmT(V)$, denoted by $\zeta_{\rmT(V)}$ is defined to be the map composition
$$\zeta_{\rmT(V)}=\zeta_{\Sym}\circ\Phi_{\rmT(V)}.$$

Equivalently, $\zeta_{\rmT(V)}$ can be defined, on generators and extended linearly and multiplicatively, as
$$\begin{array}{cccc}
\zeta_{\rmT(V)}:&\rmT(V)&\rightarrow& \mathbb{C}\\
& v_i & \mapsto & \cfrac{1}{\deg(v_i)}.
\end{array}$$

\subsection{The symmetric algebra}

Let ${\rm I}(V)$ be the commutator ideal of $\rmT(V)$ i.e. ${\rm I}(V)=\langle v_iv_j-v_jv_i \rangle$ for all pairs of basis elements. Then, the symmetric algebra on $V$ is defined as
$$\Sym(V)=\frac{\rmT(V)}{{\rm I}(V)}=\mathbb{C}\text{-span}\{\overline{v_\alpha}=v_\alpha+{\rm I}(V)\}.$$
We will use $\overline{v_\alpha}$ to denote the equivalence class of $v$ throughout this paper.

The quotient space $\Sym(V)$ is a commutative and cocommutative Hopf algebra. Its antipode $\mathcal{S}$ maps each $\overline{v_{\alpha}}$ to $(-1)^{\ell(\alpha)} \overline{v_{\alpha}}$. Since the ideal ${\rm I}(V)$ is homogeneous, $\Sym(V)$ inherits a grading from $\rmT(V)$ given by $\deg(\overline{v_\alpha})=\deg(v_\alpha)$.

It is easy to see that ${\rm I}(V)\subseteq\ker\Phi_{\rmT(V)}$. We have the induced Hopf morphism defined on generators and extended linearly and multiplicatively, as
$$\begin{array}{cccc}
\Phi_{\Sym(V)}:&\Sym(V)&\rightarrow& \Sym\\
& \overline{v_i} & \mapsto & p_{\deg(v_i)}.
\end{array}$$

The \textit{canonical} character of $\Sym(V)$, denoted by $\zeta_{\Sym(V)}$ is defined to be the induced character
$$\begin{array}{cccc}
\zeta_{\Sym(V)}:&\Sym(V)&\rightarrow& \mathbb{C}\\
& \overline{v_i} & \mapsto & \cfrac{1}{\deg(v_i)}.
\end{array}$$

For a composition $\alpha$, let $m_i$ be the number of parts in $\alpha$ that are equal to $i$ i.e. $m_i=|\{k:\alpha_k=i\}|$, and let $z_\alpha=m_1!1^{m_1}m_2!2^{m_2}\cdots$. 

\begin{theorem}\label{thm:duality}
	The symmetric algebra $\Sym(V)$ is self-dual under the Hopf pairing
	$$\left\langle \cfrac{\deg(v_\alpha)!}{\sqrt{z_\alpha}}\overline{v_\alpha},\cfrac{\deg(v_\beta)!}{\sqrt{z_\beta}}\overline{v_\beta} \right\rangle=\delta_{{\rm sort}(\alpha),{\rm sort}(\beta)}$$
	where $\deg(v_\alpha)!=\deg(\alpha_1)!\deg(\alpha_2)!\cdots\deg(\alpha_{\ell(\alpha)})!$, $\deg(v_\beta)!=\deg(\beta_1)!\deg(\beta_2)!\cdots\deg(\beta_{\ell(\beta)})!$ and sort$(\alpha)$ is the composition obtained by rearranging parts of $\alpha$ in decreasing order.
\end{theorem}

It is not hard to see that the Hopf pairing is well-defined, i.e., does not depend on the choice of representative. The theorem above follows from the isomorphism between  Hopf algebras
$$\begin{array}{ccc}
\Sym(V) & \to & \Sym\\
\overline{v_\alpha} & \mapsto & p_\alpha
\end{array}$$
and the well-known Hopf pairing on scaled power sum basis
$$\left\langle \cfrac{\prod_i\alpha_i}{\sqrt{z_\alpha}}p_\alpha,\cfrac{\prod_i\beta_i}{\sqrt{z_\beta}}p_\beta \right\rangle=\delta_{{\rm sort}(\alpha),{\rm sort}(\beta)}.$$

\subsection{The Hopf subalgebras of odd and even elements}

We call an element 
$v_{(\alpha_1,\dots,\alpha_n)}$ of $V^{\otimes n}$ {\it odd} if the degree of each $v_{\alpha_i}$ is odd, and {\it even} if the degree of each $v_{\alpha_i}$ is even.

Let the subspace of odd elements of $V$ be
$$\rmO(V)=\mathbb{C}\text{-span}\{v_i: \deg(v_i)\text{~is odd}  \},$$
and let the subspace of even elements be
$$\rmE(V)=\mathbb{C}\text{-span}\{v_i: \deg(v_i)\text{~is even}  \}.$$

\begin{proposition}
	(1) The shuffle algebras
	$$\rmS(\rmO(V))=\mathbb{C}\text{\rm-span}\{v_\alpha: v_\alpha \text{~is odd}  \}$$
	and
	$$\rmS({\rm E}(V))=\mathbb{C}\text{\rm-span}\{v_\alpha: v_\alpha \text{~is even}  \}$$ 
	are Hopf subalgebras of $\rmS(V)$.
	
	(2) The tensor algebras
	$$\rmT(\rmO(V))=\mathbb{C}\text{\rm-span}\{v_\alpha: v_\alpha \text{~is odd}  \}$$
	and
	$$\rmT({\rm E}(V))=\mathbb{C}\text{\rm-span}\{v_\alpha: v_\alpha \text{~is even}  \}$$ 
	are Hopf subalgebras of $\rmT(V)$.
	
	(3) The symmetric algebras
	$$\Sym(\rmO(V))=\mathbb{C}\text{\rm-span}\{\overline{v_\alpha}: v_\alpha \text{~is odd}  \}$$
	and
	$$\Sym({\rm E}(V))=\mathbb{C}\text{\rm-span}\{\overline{v_\alpha}: v_\alpha \text{~is even}  \}$$ 
	are Hopf subalgebras of $\Sym(V)$.
\end{proposition} 
\begin{proof}
	It follows from this fact that all these subspaces are closed under product and coproduct operations.
\end{proof} 

We call $\rmS(\rmO(V))$, $\rmT(\rmO(V))$ and $\Sym(\rmO(V))$ the {\it Hopf subalgebras of odd elements} of $\rmS(V)$, $\rmT(V)$ and $\Sym(V))$, respectively. We call $\rmS(\rmE(V))$, $\rmT(\rmE(V))$ and $\Sym(\rmE(V))$ the {\it Hopf subalgebras of even elements} of $\rmS(V)$, $\rmT(V)$ and $\Sym(V)$, respectively. In this paper, we mostly study the Hopf subalgebras of odd elements since later we show that they are the images of certain theta maps.

\begin{corollary}\label{coro:duality}
	The Hopf algebras $\rmS(\rmO(V))$ and $\rmT(\rmO(V))$ are graded dual of each other with Hopf pairing $\langle v_\alpha,v_\beta \rangle=\delta_{\alpha,\beta}$ where $\delta_{\alpha,\beta}$ is the Kronecker delta. Additionally, $\Sym(\rmO(V))$ is a self-dual Hopf algebra under the Hopf pairing $\left\langle \cfrac{\deg(v_\alpha)!}{\sqrt{z_\alpha}}\overline{v_\alpha},\cfrac{\deg(v_\beta)!}{\sqrt{z_\beta}}\overline{v_\beta} \right\rangle=\delta_{{\rm sort}(\alpha),{\rm sort}(\beta)}.$
\end{corollary}

\begin{proof}
	Both $\rmS(\rmO(V))$ and $\rmT(\rmO(V))$ are Hopf subalgebras, their basis have the same index set and the Hopf pairing follows from the duality between $\rmS(V)$ and $\rmT(V)$. Similarly, the self-duality and Hopf pairing of $\Sym(O(V))$ are inherited from $\Sym(V)$, as in Theorem \ref{thm:duality}.
\end{proof}

We then present the connection between Hopf subalgebras of odd and even elements with the odd and even Hopf subalgebras.

\begin{proposition}
	We have the following inclusions
	\begin{enumerate}
		\item $\rmS(\rmO(V)) \subseteq \sfS_{-}(\rmS(V),\zeta_{\rmS(V)})$, $\rmS(\rmE(V)) \subseteq \sfS_{+}(\rmS(V),\zeta_{\rmS(V)})$,
		\item $\rmT(\rmO(V)) \subseteq \sfS_{-}(\rmT(V),\zeta_{\rmT(V)})$, $\rmT(\rmE(V)) \subseteq \sfS_{+}(\rmT(V),\zeta_{\rmT(V)})$,
		\item $\Sym(\rmO(V)) \subseteq \sfS_{-}(\Sym(V),\zeta_{\Sym(V)})$, $\Sym(\rmE(V)) \subseteq \sfS_{+}(\Sym(V),\zeta_{\Sym(V)})$.
	\end{enumerate}
\end{proposition}

\begin{proof}
	Recall from Theorem \ref{thm:ABSodd} that $v_\alpha\in \sfS_-(\rmS(V),\zeta_{\rmS(V)})$ if and only if 
	$$({\rm id} \otimes \overline{\zeta_{\rmS(V)}^{-1}} - \zeta_{\rmS(V)} \otimes {\rm id})\circ \Delta^{(2)}(v_\alpha)=0.$$
	
	Proof of (1). Fix a word $v_\alpha=v_{\alpha_1}\cdots v_{\alpha_{\ell}}$ such that $\deg(v_{\alpha_i})=\beta_i$ is odd for all $i$, let $\beta=(\beta_1,\dots,\beta_\ell)$. Since the coproduct in $\rmS(V)$ is given by deconcatenation, it suffices to show that $\left(\overline{\zeta_{\rmS(V)}^{-1}} - \zeta_{\rmS(V)}\right)(v_\alpha)=0$.
	
	By definition, $\zeta_{\rmS(V)}=\zeta_\QSym\circ\Phi_{\rmS(V)}$ and $\Phi_{\rmS(V)}(v_\alpha)=S_\beta$, we have
	$$\left(\overline{\zeta_{\rmS(V)}^{-1}} - \zeta_{\rmS(V)}\right)(v_\alpha)=\overline{\zeta_\QSym^{-1}}(S_\beta)-\zeta_\QSym(S_\beta)=\frac{(-1)^{\deg(v_\alpha)+\ell}}{\ell!}-\frac{1}{\ell!}=0$$
	where the second to last equality follows from equation \ref{eq:3.1} and \ref{eq:3.2}. The even parts can be proved using similar arguments.
	
	Proof of (2). Since $\rmT(V)$ is freely generated by $\{v_i\}$, and by Theorem \ref{thm:ABSodd}, $S_-(\rmT(V),\zeta_{\rmT(V)})$ is an algebra, it suffices to show that $v_i\in S_-(\rmT(V),\zeta_{\rmT(V)})$ for all $v_i$ with odd degree. Since $v_i$ is primitive in $\rmT(V)$, we have $S(v_i)=-v_i$ and hence $\overline{\zeta_{\rmT(V)}^{-1}}(v_i)=\zeta_{\rmT(V)}(v_i)$. Therefore, $\left(\overline{\zeta_{\rmT(V)}^{-1}} - \zeta_{\rmT(V)}\right)(v_i)=0$ and $v_i\in S_-(\rmT(V),\zeta_{\rmT(V)})$ follows from
	$$({\rm id} \otimes \overline{\zeta_{\rmT(V)}^{-1}} - \zeta_{\rmT(V)} \otimes {\rm id})\circ \Delta^{(2)}(v_i)=0.$$
	
	The even part can be proved using similar arguments. 
	
	Proof of (3). Since $\Sym(V)$ is generated by $\overline{v_i}$ with odd degree, the last inclusions can be proved using a similar argument as well.
\end{proof}

\subsection{The theta maps}

Consider the linear maps defined on basis elements and extended linearly
\begin{align*}
\begin{array}{cccc} 
\Theta_{\rmS(V)}:&\rmS(V)&\rightarrow& \rmS(V),\\
& v_\alpha &\mapsto& \begin{cases} 
2^{\ell(\alpha)}v_\alpha& \text{if $\alpha$ is odd}\\
0 & \text{otherwise.}
\end{cases}
\end{array}
\end{align*}
\begin{align*}
\begin{array}{cccc} 
	\Theta_{\rmT(V)}:&\rmT(V)&\rightarrow& \rmT(V),\\
	& v_\alpha &\mapsto& \begin{cases} 
		2^{\ell(\alpha)}v_\alpha& \text{if $\alpha$ is odd}\\
		0 & \text{otherwise.}
	\end{cases}
\end{array}
\end{align*}
\begin{align*}
\begin{array}{cccc} 
	\Theta_{\Sym(V)}:&\Sym(V)&\rightarrow& \Sym(V),\\
	& \overline{v_\alpha} &\mapsto& \begin{cases} 
		2^{\ell(\alpha)}\overline{v_\alpha}& \text{if $\alpha$ is odd}\\
		0 & \text{otherwise.}
	\end{cases}
\end{array}
\end{align*}

It is easy to see that the images of $\Theta_{\rmS(V)}$, $\Theta_{\rmT(V)}$ and $\Theta_{\Sym(V)}$ are $\rmS(\rmO(V))$, $\rmT(\rmO(V))$ and $\Sym(\rmO(V))$, respectively.

Next, we show that $\Theta_{\rmS(V)}$, $\Theta_{\rmT(V)}$ and $\Theta_{\Sym(V)}$ are theta maps for $(\rmS(V),\zeta_{\rmS(V)})$, $(\rmS(V),\zeta_{\rmT(V)})$ and $(\rmS(V),\zeta_{\Sym(V)})$, respectively. As a consequence, $\rmS(\rmO(V))$, $\rmT(\rmO(V))$ and $\Sym(\rmO(V))$ are peak algebras of $\rmS(V)$, $\rmT(V)$ and $\Sym(V)$, respectively.

\begin{theorem}\label{thm:odd-shuffle}
	The map $\Theta_{\rmS(V)}$ is a graded Hopf morphism and moreover, the following diagram commutes. 
	$$
	\begin{tikzpicture}
	\node(TV) at (0,0){$\rmS(V)$};
	\node(Sym) at (4,0){$\QSym$};
	\node(OV) at (0,-2){$\rmS(\rmO(V))$};
	\node(Omega) at (4,-2){$\Pi$};
	
	\draw[thick,->]  (TV)->(Sym);
	\draw[thick,->]  (OV)->(Omega);
	\draw[thick,->]   (TV)->(OV);
	\draw[thick,->]   (Sym)->(Omega);
	
	\node(Phi) at (2,0.2){$\Phi_{\rmS(V)}$};
	\node(Phi) at (2,-1.8){$\Phi_{\rmS(V)}$};
	\node(thetaV) at (-0.5,-1){$\Theta_{\rmS(V)}$};
	\node(thetaV) at (4.7,-1){$\Theta_\QSym$};
	\end{tikzpicture} 
	$$
	So the map $\Theta_{\rmS(V)}$ is a theta map for $(\rmS(V),\zeta_{\rmS(V)})$.
\end{theorem} 

\begin{proof}
	The map $\Theta_{\rmS(V)}$ is an algebra morphism since for odd compositions $\alpha$ and $\beta$,
	$$\Theta_{\rmS(V)}(v_\alpha)\Theta_{\rmS(V)}(v_\beta)=2^{\ell(\alpha)}v_\alpha\cdot2^{\ell(\beta)} v_\beta=2^{\ell(\alpha)+\ell(\beta)}\sum_{\gamma\in\alpha\shuffle\beta}v_{\gamma}=\Theta_{\rmS(V)}(v_\alpha\cdot v_\beta)$$
	and both sides vanish if at least one of $\alpha$ and $\beta$ is not odd.
	
	The map $\Theta_{\rmS(V)}$ is a coalgebra morphism since for odd composition $\alpha$,
	$$\Delta(\Theta_{\rmS(V)}(v_\alpha))=\Delta\left(2^{\ell(\alpha)}v_\alpha\right)=2^{\ell(\alpha)}\sum_{\beta\gamma=\alpha}v_{\beta}\otimes v_{\gamma}=\left(\Theta_{\rmS(V)}\otimes\Theta_{\rmS(V)}\right)\circ\Delta(v_\alpha)$$
	and both sides vanish if $\alpha$ is not odd.
	
	By Theorem \ref{thm:theta}, we have
	$$\Theta_\QSym\circ \Phi_{\rmS(V)}(v_\alpha)=\Theta_\QSym(S_\alpha)=\left\{\begin{array}{cc}
	2^{\ell(\alpha)}S_\alpha & \text{if }\alpha \text{ is odd}\\
	0& \text{otherwise.}
	\end{array}\right.$$
	On the other hand,
	$$\Phi_{\rmS(V)}\circ\Theta_{\rmS(V)}(v_\alpha)=\left\{\begin{array}{cc}
	\Phi_{\rmS(V)}\left(2^{\ell(\alpha)}v_\alpha\right) & \text{if }\alpha \text{ is odd}\\
	0& \text{otherwise}
	\end{array}\right.=\left\{\begin{array}{cc}
	2^{\ell(\alpha)}S_\alpha & \text{if }\alpha \text{ is odd}\\
	0& \text{otherwise.}
	\end{array}\right.$$
	Hence, the diagram commutes.
\end{proof}

\begin{theorem}\label{thm:odd-tensor}
	The map $\Theta_{\rmT(V)}$ is a graded Hopf morphism and moreover, the following diagram commutes. 
	$$
	\begin{tikzpicture}
	\node(TV) at (0,0){$\rmT(V)$};
	\node(Sym) at (4,0){$\Sym$};
	\node(OV) at (0,-2){$\rmT(\rmO(V))$};
	\node(Omega) at (4,-2){$\Omega$};
	
	\draw[thick,->]  (TV)->(Sym);
	\draw[thick,->]  (OV)->(Omega);
	\draw[thick,->]   (TV)->(OV);
	\draw[thick,->]   (Sym)->(Omega);
	
	\node(Phi) at (2,0.2){$\Phi_{\rmT(V)}$};
	\node(Phi) at (2,-1.8){$\Phi_{\rmT(V)}$};
	\node(thetaV) at (-0.55,-1){$\Theta_{\rmT(V)}$};
	\node(thetaV) at (4.6,-1){$\Theta_\Sym$};
	\end{tikzpicture} 
	$$
	So the map $\Theta_{\rmT(V)}$ is a theta map for $(\rmT(V),\zeta_{\rmT(V)})$.
\end{theorem} 

\begin{proof}
	The map $\Theta_{\rmT(V)}$ is an algebra morphism since for odd compositions $\alpha$ and $\beta$,
	$$\Theta_{\rmT(V)}(v_\alpha)\Theta_{\rmT(V)}(v_\beta)=2^{\ell(\alpha)}v_\alpha\cdot2^{\ell(\beta)} v_\beta=2^{\ell(\alpha)+\ell(\beta)}v_{\alpha\beta}=\Theta_{\rmT(V)}(v_\alpha\cdot v_\beta)$$
	and both sides vanish if at least one of $\alpha$ and $\beta$ is not odd.
	
	The map $\Theta_{\rmT(V)}$ is a coalgebra morphism since for odd composition $\alpha$,
	$$\Delta(\Theta_{\rmT(V)}(v_\alpha))=\Delta\left(2^{\ell(\alpha)}v_\alpha\right)=2^{\ell(\alpha)}\sum_{J\subseteq[\ell(\alpha)]}v_{\alpha_J}\otimes v_{\alpha_{[\ell(\alpha)]\setminus J}}=\left(\Theta_{\rmT(V)}\otimes\Theta_{\rmT(V)}\right)\circ\Delta(v_\alpha)$$
	and both sides vanish if $\alpha$ is not odd.
	
	We have that
	$$\Theta_\Sym\circ \Phi_{\rmT(V)}(v_\alpha)=\Theta_\Sym(p_\alpha)=\left\{\begin{array}{cc}
	2^{\ell(\alpha)}p_\alpha & \text{if }\alpha \text{ is odd}\\
	0& \text{otherwise.}
	\end{array}\right.$$
	On the other hand,
	$$\Phi_{\rmT(V)}\circ\Theta_{\rmT(V)}(v_\alpha)=\left\{\begin{array}{cc}
	\Phi_{\rmT(V)}\left(2^{\ell(\alpha)}v_\alpha\right) & \text{if }\alpha \text{ is odd}\\
	0& \text{otherwise}
	\end{array}\right.=\left\{\begin{array}{cc}
	2^{\ell(\alpha)}p_\alpha & \text{if }\alpha \text{ is odd}\\
	0& \text{otherwise.}
	\end{array}\right.$$
	Hence, the diagram commutes.
\end{proof}

\begin{theorem}\label{thm:odd-symmetric}
	The map $\Theta_{\Sym(V)}$ is a graded Hopf morphism and moreover, the following diagram commutes. 
	$$
	\begin{tikzpicture}
	\node(TV) at (0,0){$\Sym(V)$};
	\node(Sym) at (4,0){$\Sym$};
	\node(OV) at (0,-2){$\Sym(\rmO(V))$};
	\node(Omega) at (4,-2){$\Omega$};
	
	\draw[thick,->]  (TV)->(Sym);
	\draw[thick,->]  (OV)->(Omega);
	\draw[thick,->]   (TV)->(OV);
	\draw[thick,->]   (Sym)->(Omega);
	
	\node(Phi) at (2,0.2){$\Phi_{\Sym(V)}$};
	\node(Phi) at (2,-1.8){$\Phi_{\Sym(V)}$};
	\node(thetaV) at (-0.75,-1){$\Theta_{\Sym(V)}$};
	\node(thetaV) at (4.55,-1){$\Theta_\Sym$};
	\end{tikzpicture} 
	$$
	and so the map $\Theta_{\Sym(V)}$ is a theta map for $(\Sym(V),\zeta_{\Sym(V)})$.
\end{theorem} 

\begin{proof}
	The same arguments as in the proof of Theorem \ref{thm:odd-tensor} can be applied.
\end{proof}

\begin{theorem}\label{thm:adjoint}
	The maps $\Theta_{\rmS(V)}$ and $\Theta_{\rmT(V)}$ are adjoint maps of each other. Additionally, the map $\Theta_{\Sym(V)}$ is self-adjoint.
\end{theorem}

\begin{proof}
	By definition of $\Theta_{\rmS(V)}$ and $\Theta_{\rmT(V)}$ and the duality of $\rmS(V)$ and $\rmT(V)$, we have
	$$\langle\Theta_{\rmS(V)}(v_\alpha),v_\beta \rangle=\left\{\begin{array}{cc}
	2^{\ell(\alpha)} & \text{if }\alpha=\beta \text{ and }\alpha $\text{ is odd}$\\
	0 & \text{otherwise}
	\end{array}\right.$$
	and
	$$\langle v_\alpha,\Theta_{\rmT(V)}(v_\beta) \rangle=\left\{\begin{array}{cc}
	2^{\ell(\alpha)} & \text{if }\alpha=\beta \text{ and }\alpha $\text{ is odd}$\\
	0 & \text{otherwise.}
	\end{array}\right.$$
	Similar argument works for $\Theta_{\Sym(V)}$.
\end{proof}

\section{Examples of theta maps and peak algebras}\label{sec:perms}

\subsection{The classical cases}
Let $V$ be a vector space with a distinguished basis $\{v_1,v_2,\dots\}$ and consider the grading $\deg(v_i)=i$.

Then we have the following isomorphisms of graded Hopf algebras,
\begin{align*}
	\rmS(V)&\cong\QSym\\
	v_\alpha&\mapsto S_\alpha,\\
	\rmT(V)&\cong\NSym\\
	v_\alpha&\mapsto S_\alpha^*,\\
	\Sym(V)&\cong\Sym\\
	\overline{v_\alpha}&\mapsto p_\alpha.
\end{align*}

Via these isomorphisms, $\Theta_{\rmS(V)}$ and $\Theta_{\Sym(V)}$ can be identified as the maps $\Theta_\QSym$ and $\Theta_\Sym$, respectively. Additionally, the Hopf subalgebras of odd elements $\rmS(\rmO(V))$ and $\Sym(\rmO(V))$ can be identified as $\Pi$ and $\Omega$, respectively.

On the other hand, the theta map $\Theta_{\rmT(V)}$ can be identified as the map $\Theta_\NSym:\NSym\to\NSym$ that

\begin{align*}
	S_\alpha^* \mapsto \left\{\begin{array}{cc}
		2 S_\alpha^*&\text{ if }\alpha\text{ is odd}\\
		0&\text{otherwise.}
	\end{array}\right.
\end{align*}

The map $\Theta_\NSym$ is also defined in \cite{N03,S05} in the group algebra $\mathbb{C}\mathfrak{S}$ using descent and peak sets. The equivalence can be deduced from the fact that $\Theta_\NSym$ is the adjoint map of $\Theta_\QSym$.

In particular, we recover the following results

\begin{corollary}
	The Hopf algebras $\Pi$ and $\Omega$ are shuffle algebra and symmetric algebra, respectively. Hence, $\Omega$ is self-dual. Moreover, $\img(\Theta_\NSym)\cong\Pi^*$.
\end{corollary}

The relations can be seen more clearly from the following commutative diagram. The rectangle on the left is dual to the rectangle on the right.

\begin{center}
	\begin{tikzpicture}
		\node(TV) at (0,0){$\Sym^*$};
		\node(Sym) at (4,0){$\Sym$};
		\node(OV) at (0,-2){$\Sym^*$};
		\node(Omega) at (4,-2){$\Sym$};
		\node(QSym1) at (8,0){$\QSym$};
		\node(QSym2) at (8,-2){$\QSym$};
		\node(NSym1) at (-4,0){$\NSym$};
		\node(NSym2) at (-4,-2){$\NSym$};
		
		\draw[thick,<->]  (TV)->(Sym);
		\draw[thick,<->]  (OV)->(Omega);
		\draw[thick,->]   (TV)->(OV);
		\draw[thick,->]   (Sym)->(Omega);
		\draw[thick,->]   (Sym)->(QSym1);
		\draw[thick,->]   (Omega)->(QSym2);
		\draw[thick,->]   (QSym1)->(QSym2);
		\draw[thick,->>]   (NSym1)->(TV);
		\draw[thick,->>]   (NSym2)->(OV);
		\draw[thick,->]   (NSym1)->(NSym2);
		
		\node(Phi) at (2,0.2){$\cong$};
		\node(Phi) at (2,-1.8){$\cong$};
		\node(thetaV) at (-0.5,-1){$\Theta_\Sym^*$};
		\node(thetaV) at (4.55,-1){$\Theta_\Sym$};
		\node at (-2,0.2) {$I^{-1}\circ\pi$};
		\node at (-2,-1.8) {$I^{-1}\circ\pi$};
		\node at (6,0.2) {$\iota$};
		\node at (6,-1.8) {$\iota$};
		\node(thetaV) at (-4.7,-1){$\Theta_\NSym$};
		\node(thetaV) at (8.7,-1){$\Theta_\QSym$};
	\end{tikzpicture}
\end{center}

\subsection{The Hopf algebra $\mathcal{V}$}

As is shown in \cite{A05}, the associated graded Hopf algebra for any graded cofree coalgebra is a shuffle algebra. In this subsection we present a co-commutative Hopf structure, $\mathcal{V}$, on permutations that is the associated graded Hopf algebra to the Malvenuto–Reutenauer Hopf algebra. This Hopf algebra is isomorphic to Grossman-Larson Hopf algebra \cite{G89} and also appears in \cite{L21,V14}

Fix a permutation $\sigma\in\mathfrak{S}_n$, its global descent set $\rm{GD}(\sigma)$ is defined to be $\{i:\sigma(a)>\sigma(b)\text{ for all }a\leq i<b\}$. For example, $\rm{GD}(5763421)=\{3,5,6\}$. The shifted concatenation of two permutations $\sigma_1\cdots\sigma_n$ and $\tau_1\cdots\tau_m$ is defined as $\sigma\odot\tau=(\sigma_1+m)\cdots(\sigma_n+m)\tau_1\cdots\tau_m$. For example, $132\odot 3421=5763421$. For each permutation $\sigma$, there is a unique decomposition
$$\sigma=\sigma^1\odot\sigma^2\odot\cdots\odot\sigma^\ell$$
such that $\sigma^i$ are permutations and $\rm{GD}(\sigma^i)=\emptyset$. For example, $5763421=132\odot12\odot1\odot1$. In this subsection, we view a permutation as a word whose letters are permutations without global descent. The degree of a permutation $\sigma$, denoted by $\deg(\sigma)$, is $n$ if $\sigma\in\fS_n$.

As a graded vector space, $\mathcal{V}=\displaystyle\bigoplus_{n\geq 0}\mathcal{V}_n$ where $\mathcal{V}_n=\mathbb{C}$-span$\{\sigma\in\mathfrak{S}_n\}$. For convenience, $\fS_0$ contains the unique empty permutation $\emptyset$. The product is given by shuffle of permutations
$$\sigma\cdot\tau=\sum_{\eta\in\sigma\shuffle\tau}\eta,$$
and the coproduct is given by deconcatenation
$$\Delta(\sigma)=\sum_{\tau\odot\eta=\sigma}\tau\otimes\eta.$$
For example,
\begin{align*}
	{\color{red}132\odot 12}\cdot{\color{blue}12\odot 1}&={\color{red}132}\odot{\color{red}12}\odot{\color{blue}12}\odot{\color{blue}1}+{\color{red}132}\odot{\color{blue}12}\odot{\color{red}12}\odot{\color{blue}1}+{\color{red}132}\odot{\color{blue}12}\odot{\color{blue}1}\odot{\color{red}12}\\
	&+{\color{blue}12}\odot{\color{red}132}\odot{\color{red}12}\odot{\color{blue}1}+{\color{blue}12}\odot{\color{red}132}\odot{\color{blue}1}\odot{\color{red}12}+{\color{blue}12}\odot{\color{blue}1}\odot{\color{red}132}\odot{\color{red}12},
\end{align*}
$$\Delta(132\odot 12\odot 1)=\emptyset\otimes 132\odot 12\odot 1+132\otimes 12\odot 1+132\odot 12\otimes 1+132\odot 12\odot 1\otimes\emptyset.$$
With this Hopf structure, $\cV$ is a shuffle algebra. The linear map defined on basis elements
$$\begin{array}{cccc}
	\Phi_{\cV}:&\cV&\rightarrow& \QSym\\
	&\sigma^1\odot\sigma^2\odot\cdots\odot\sigma^\ell& \mapsto & S_{(\deg(\sigma^1),\dots,\deg(\sigma^\ell))}
\end{array}$$
makes $\cV$ a combinatorial Hopf algebra. Then, Theorem \ref{thm:odd-shuffle} describes a theta map for $\cV$ as follows.

\begin{corollary}
The linear map defined on basis elements as
$$
\begin{array}{cccc} 
	\Theta_\cV:&\cV&\rightarrow& \cV,\\
	& \sigma^1\odot\sigma^2\odot\cdots\odot\sigma^\ell &\mapsto& \begin{cases} 
		2^\ell (\sigma^1\odot\sigma^2\odot\cdots\odot\sigma^\ell) & \text{if $\deg(\sigma^i)$ is odd for all $i$}\\
		0 & \text{otherwise.}
	\end{cases}
\end{array}
$$
is a theta map for $\cV$.
\end{corollary}

The peak algebra for $\mathcal{V}$ is given by $\Theta_{\mathcal{V}}(\mathcal{V})=\mathbb{C}$-span$\{\sigma^1\odot\cdots\odot\sigma^\ell:\rm{GD}(\sigma^i)=\emptyset,\deg(\sigma^i)\text{ are odd}\}$. Its Hilbert series is given by
$$\rm{Hilb}(\Theta_{\mathcal{V}}(\mathcal{V}))=1+x+x^2+4x^3+7x^4+81x^5+164x^6+\cdots.$$

\subsection{The Hopf algebra $\NCSym$}

The Hopf algebra of symmetric functions in noncommuting variables ${\rm NCSym}=\bigoplus_{n\geq 0}{\rm NCSym}_n$ are originally defined in \cite{W36}. For more information about what algebraic structures of ${\rm NCSym}$ has been revealed  see \cite[Introduction]{A21}. 

A set partition of $[n]$ is a set of disjoint non-empty subsets of $[n]$ whose union is $[n]$.  A set partition of $[n]$ is said to be atomic if there exists $i$ such that for any $a\leq i<b$, $a$ and $b$ are not in the same subset. The bases of ${\rm NCSym}_n$ are indexed by set partitions of $[n]$. It is known that $\NCSym$ is free as an algebra, whose generators are indexed by atomic set partitions. Lauve and Mastnak \cite{L11} discovered that there is a set of primitive elements $\{\mathbf{p}_\pi:\pi \text{~atomic}\}$ that freely generates ${\rm NCSym}$, in particular, $\NCSym$ is a tensor algebra.

The degree of $\pi$ is $n$ if $\pi$ is a set partition of $[n]$. As a consequence of Theorem \ref{thm:odd-tensor}, the linear map defined on basis elements as
$$\begin{array}{cccc}
\Theta_{\NCSym}: & \NCSym& \rightarrow & {\rm NCPeak}\\
& \mathbf{p}_{\pi_1}\cdots\mathbf{p}_{\pi_\ell} & \mapsto & \begin{cases}
2 \mathbf{p}_{\pi_1}\cdots\mathbf{p}_{\pi_\ell} & \deg(\pi_i)\text{ are odd},\\
0 & \text{otherwise.}
\end{cases}  \end{array}$$
is a theta map for $\NCSym$.

The peak algebra for $\NCSym$ is given by $\Theta_\NCSym(\NCSym)=\mathbb{C}$-span$\{\mathbf{p}_{\pi_1}\cdots\mathbf{p}_{\pi_\ell}:\deg(\pi_i)\text{ are odd}\}$. Its Hilbert series is given by
$$\rm{Hilb}(\Theta_{\NCSym}(\NCSym))=1+x+x^2+3x^3+5x^4+29x^5+57x^6+\cdots.$$


\end{document}